\documentclass[11pt,a4paper]{amsart}
\usepackage{amssymb,amsmath,amsfonts}

\headsep=1.2500cm \numberwithin{equation}{section}
\newtheorem{Theorem}{Theorem}[section]
\newtheorem{Proposition}[Theorem]{Proposition}
\newtheorem{cor}[Theorem]{Corollary}
\newtheorem{lemma}[Theorem]{Lemma}
\theoremstyle{remark}
\newtheorem{Definition}[Theorem]{Definition}

\begin{document}

\title{A note on left $\phi$-biflat Banach algebras}

  \author[A. Sahami]{A. Sahami}
  \email{a.sahami@ilam.ac.ir}

  \address{Department of Mathematics
    Faculty of Basic Sciences Ilam University P.O. Box 69315-516 Ilam,
    Iran.}

 \author[M. Rostami]{M. Rostami}
\email{mross@aut.ac.ir}

\address{Faculty of Mathematics and Computer Science
	Amirkabir University of Technology
	424 Hafez Avenue, 15914 Tehran
	Iran}

 \author[A. Pourabbas]{A. Pourabbas}
\email{arpabbas@aut.ac.ir}

\address{Faculty of Mathematics and Computer Science
	Amirkabir University of Technology
	424 Hafez Avenue, 15914 Tehran
	Iran}



\maketitle

\begin{abstract}
In this paper, we study the notion of $\phi$-biflatness for some Banach algebras, where $\phi$ is a non-zero multiplicative linear functional. We  show that the Segal algebra $S(G)$ is left $\phi$-biflat if and only if $G$ is amenable. Also we characterize left $\phi$-biflatness of semigroup algebra $\ell^{1}(S)$ in the term of biflatness, where $S$ is a Clifford semigroup. 
\end{abstract}
\section{Introduction and preliminaries}
A Banach algebra $A$ is called amenable, if there exists an element $M\in (A\otimes_{p}A)^{**}$ such that $a\cdot M=M\cdot a$ and $\pi^{**}_{A}(M)a=a$ for each $a\in A.$ It is well-known that an amenable Banach algebra has a bounded approximate identity. For the history of amenability see \cite{run}.

In homological theory, the notion of biflatness is an amenability-like property.
In fact a Banach algebra $A$ is biflat if there exists a Banach $A$-bimodule $\rho$ from $A$ into $(A\otimes_{p}A)^{**}$ such that $\pi^{**}_{A}\circ\rho(a)=a,$ for each $a\in A.$ It is well-known that a  Banach algebra $A$ with a bounded approximate identity is biflat if and only if $A$ is amenable.

Kanuith et. al. in \cite{kan}, defined a version of amenability with respect to a non-zero multiplicative functional $\phi$. Indeed a Banach algebra $A$ is called left $\phi$-amenable if there exists an element $m\in A^{**}$ such that $am=\phi(a)m$ and $\tilde{\phi}(m)=1$ for every $a\in A.$ We shall mention that  the Segal algebra $S(G)$ is left $\phi$-amenable if and only if $G$ is amenable, for further information see \cite{Ala},\cite{jav} and \cite{Hu}.

Motivated by these considerations, Essmaili et. al. in \cite{rost ch} defined a biflat-like property related to a multiplicative linear functional, they called it condition $W$ (which we call it here right $\phi$-biflatness). 
\begin{Definition}\cite{rost ch}
Let $A$ be a Banach algebra and $\phi\in\Delta(A)$. The Banach algebra $A$ is called left $\phi$-biflat (right $\phi$-biflat or satisfies condition $W$), if there exists a bounded linear map $\rho:A\rightarrow (A\otimes_{p}A)^{**}$
such that $$\rho(ab)=\phi(b)\rho(a)=a\cdot \rho(b)\quad (\rho(ab)=\phi(a)\rho(b)=\rho(a)\cdot b)$$
and $$\tilde{\phi}\circ\pi^{**}_{A}\circ\rho(a)=\phi(a),$$
for each $a,b\in A,$ respectively.
\end{Definition}
They showed that  a symmetric Segal algebra $S(G)$ (on a locally compact group $G$) is right $\phi$-biflat if and only if $G$ is amenable \cite[Theorem 3.4]{rost ch}. As a consequence of this result in \cite[Corollary 3.5]{rost ch} authors charactrized the right $\phi$-biflatness of Lebesgue-Fourier algebra $\mathcal{LA}(G)$, Weiner algebra $M_{1}$ and Feichtinger's Segal algebra $S_{0}(G)$ over a unimodular locally compact group.

In this paper, we extend   \cite[Theorem 3.4]{rost ch}  for any Segal algebra (in left $\phi$-biflat case). In fact we show that the Segal algebra $S(G)$ is left $\phi$-biflat if and only if $G$ is amenable. Using this tool we charactrize left $\phi$-biflatness of the
Lebesgue-Fourier algebra $\mathcal{LA}(G)$.
 Also we characterize left $\phi$-biflatness of second dual of Segal algebra $S(G)^{**}$ in the term of amenability $G.$ We  study left $\phi$-biflatness of some semigroup algebras.

We remark some standard notations and definitions that we shall need
in this paper. Let $A$ be a Banach algebra. If $X$ is a Banach
$A$-bimodule, then  $X^{*}$ is also a Banach $A$-bimodule via the
following actions
$$(a\cdot f)(x)=f(x\cdot a) ,\hspace{.25cm}(f\cdot a)(x)=f(a\cdot x ) \hspace{.5cm}(a\in A,x\in X,f\in X^{*}). $$

Throughout, the
character space of $A$ is denoted by $\Delta(A)$,  that is, all
non-zero multiplicative linear functionals on $A$. Let $\phi\in
\Delta(A)$. Then $\phi$ has a unique extension   $\tilde{\phi}\in\Delta(A^{**})$
which is defined by $\tilde{\phi}(F)=F(\phi)$ for every
$F\in A^{**}$.

Let $A$ be a  Banach algebra. The projective tensor product
$A\otimes_{p}A$ is a Banach $A$-bimodule via the following actions
$$a\cdot(b\otimes c)=ab\otimes c,~~~(b\otimes c)\cdot a=b\otimes
ca\hspace{.5cm}(a, b, c\in A).$$
 The product morphism $\pi_{A}:A\otimes_{p}A\rightarrow A$ is given by $\pi_{A}(a\otimes b)=ab,$ for every $a,b\in A.$
 Let $X$ and $Y$ be Banach $A$-bimodules. The map $T:X\rightarrow Y$ is called $A$-bimodule morphism, if
$$T(a\cdot x)=a\cdot T(x),\quad T(x\cdot a)=T(x)\cdot a,\qquad (a\in A,x\in X).$$

\section{Left $\phi$-biflatness}
In this section we give two criterion which show the relation of left $\phi$-biflatness
 and left $\phi$-amenability.
 \begin{lemma}\label{main}
	Suppose that $A$ is a left $\phi$-biflat Banach algebra with $\overline{A\ker\phi}^{||\cdot||}=\ker\phi.$ Then $A$ is left $\phi$-amenable.
\end{lemma}
\begin{proof}
	Let $A$ be left $\phi$-biflat. Then there exists a bounded linear map $\rho:A\rightarrow (A\otimes_{p}A)^{**}$ such that $\rho(ab)=a\cdot \rho(b)=\phi(b)\rho(a) $ and $\tilde{\phi}\circ\pi^{**}_{A}\circ\rho(a)=\phi(a)$ for all $a\in A.$ We finish the proof in three steps:
	
	Step1:There exists a bounded left $A$-module morphism $\xi:A\rightarrow (A\otimes_{p}\frac{A}{\ker\phi})^{**}$
	which $\xi(l)=0,$ for each $l\in\ker\phi.$
	To see this, we denote $id_{A}:A\rightarrow A$ for the identity map. Also we denote $q:A\rightarrow \frac{A}{\ker\phi}$ for the qoutient map. Put $\xi:=(id_{A}\otimes q)^{**}\circ\rho:A\rightarrow (A\otimes_{p}\frac{A}{\ker\phi})^{**}$, where $id_{A}\otimes q(a\otimes b)=id_{A}(a)\otimes q(b)$ for every $a,b\in A.$ Clearly $id_{A}\otimes q:A\otimes_{p}A\rightarrow A\otimes_{p}\frac{A}{\ker\phi}$ is a bounded left $A$-module morphism, it folloows that $(id_{A}\otimes q)^{**}$ is also 
	 a bounded left $A$-module morphism. So $\xi:A\rightarrow (A\otimes_{p}\frac{A}{\ker\phi})^{**}$ is a bounded left $A$-module morphism.	
	  Let $l$ be an arbitrary element of $\ker\phi$. Since $\overline{A\ker\phi}^{||\cdot||}=\ker\phi$, there exist two sequences $(a_{n})$ in $A$ and $(l_{n})$ in $\ker\phi$ such that $a_n l_{n}\xrightarrow{||\cdot||}l.$ $$\xi(l)=(id_{A}\otimes q)^{**}\circ\rho(l)=\lim_{n}(id_{A}\otimes q)^{**}\circ\rho(a_{n}l_{n})=\lim_{n}\phi(l_{n})(id_{A}\otimes q)^{**}\circ\rho(a_{n})=0,$$
	  the last equality holds because $(l_{n})$ is in $\ker\phi.$
	  
	  Step2:There exists a bounded left $A$-module morphism $\eta:\frac{A}{\ker\phi}\rightarrow A^{**}$ such that $\tilde{\phi}\circ \eta(a+\ker\phi)=\phi(a)$ for each $a\in A.$ To see this, in step1  we showed that $\xi(\ker\phi)=\{0\}$. It induces a map $\overline{\xi}:\frac{A}{\ker\phi}\rightarrow (A\otimes_{p}\frac{A}{\ker\phi})^{**}$ which is defined by $\overline{\xi}(a+\ker\phi)=\xi(a)$ for each $a\in A.$ Define $$\theta:=(id_{A}\otimes\overline{\phi})^{**}\circ\overline{\xi}:\frac{A}{\ker\phi}\rightarrow  (A\otimes_{p}\frac{A}{\ker\phi})^{**},$$
	  where $\overline{\phi}$ is a character on $\frac{A}{\ker\phi}$ given by $\overline{\phi}(a+\ker\phi)=\phi(a)$ for each $a\in A.$ Clearly $\theta$ is a bounded left $A$-module morphism. On the other hand we know that $\frac{A}{\ker\phi}\cong \mathbb{C}$  and $A\otimes_{p}\frac{A}{\ker\phi}\cong A$. Thus the composition of $\tilde{\phi}$ and $\theta$ can be defined. Since $$\tilde{\phi}\circ(id_{A}\otimes \overline{\phi})^{**}=(\phi\otimes \overline{\phi})^{**},\quad (\phi\otimes \overline{\phi})^{**}\circ\xi(a)=\tilde{\phi}\circ\pi^{**}_{A}\circ\rho(a),\qquad (a\in A),$$
	  we have 
	  	   \begin{equation*}
	  \begin{split}
	 \tilde{\phi}\circ\theta(a+\ker\phi)=\tilde{\phi}\circ(id_{A}\otimes \overline{\phi})^{**}\circ\overline{\xi}(a+\ker\phi)&=(\phi\otimes \overline{\phi})^{**}\circ\xi(a)\\
	 &=\tilde{\phi}\circ\pi^{**}_{A}\circ\rho(a)\\
	 &=\phi(a),
	  \end{split}
	  \end{equation*}
	  for each $a\in A.$
	  
	  Step3: We prove that $A$ is left $\phi$-amenable. To see that, choose an element $a_{0}$ in $A$ such that $\phi(a_{0})=1.$ Put $m=\theta(a_{0}+\ker\phi)\in A^{**}$. Since $aa_{0}-\phi(a)a_{0}\in\ker\phi,$ we have $aa_{0}+\ker\phi=\phi(a)a_{0}+\ker\phi.$ Consider $$am=a\theta(a_{0}+\ker\phi)=\theta(aa_{0}+\ker\phi)=\theta(\phi(a)a_{0}+\ker\phi)=\phi(a)\theta(a_{0}+\ker\phi)=\phi(a)m$$
	  and 
	  $$\tilde{\phi}(m)=\tilde{\phi}\circ \theta(a_{0}+\ker\phi)=\phi(a_{0})=1,$$
	  for every $a\in A.$ It implies that $A$ is left $\phi$-amenable.
\end{proof}
\begin{Theorem}\label{dual}
	Let $A$ be a Banach algebra with a left approximate identity and $\phi\in\Delta(A).$ Then $A^{**}$ is left $\tilde{\phi}$-biflat if and only if $A$ is left $\phi$-biflat.
\end{Theorem}
\begin{proof}
Suppose that $A^{**}$ is left $\tilde{\phi}$-biflat. Then there exists a bounded linear map $\rho:A^{**}\rightarrow (A^{**}\otimes_{p}A^{**})^{**}$	 such that $
\tilde{\tilde{\phi}}\circ\pi^{**}_{A^{**}}\circ\rho(a)=\tilde{\phi}(a)$ for all $a\in A^{**}.$ On the other hand, there exists a
bounded linear map $\psi:A^{**}\otimes_{p} A^{**}\rightarrow
(A\otimes_{p} A)^{**}$ such that for $a,b\in A$ and $m\in
A^{**}\otimes_{p} A^{**}$, the following holds;
\begin{enumerate}
	\item [(i)] $\psi(a\otimes b)=a\otimes b $,
	\item [(ii)] $\psi(m)\cdot a=\psi(m\cdot a)$,\qquad
	$a\cdot\psi(m)=\psi(a\cdot m),$
	\item [(iii)] $\pi_{A}^{**}(\psi(m))=\pi_{A^{**}}(m),$
\end{enumerate}
see \cite[Lemma 1.7]{gha}. Clearly $$\psi^{**}\circ\rho|_{A}:A\rightarrow (A\otimes_{p}A)^{**}$$ is a bounded linear map which 
$$\psi^{**}\circ\rho|_{A}(ab)=\phi(b)\psi^{**}\circ\rho|_{A}(a)=a\cdot \psi^{**}\circ\rho|_{A}(b)$$
and 
$$\tilde{\tilde{\phi}}\circ\pi^{****}_{A}\circ\rho(a)=\tilde{\phi}(a),\quad (a,b\in A)$$ 
Following the similar arguments as in the  previous lemma (step 1), we can find a bounded left $A$-module morphism $\xi:A\rightarrow (A\otimes_{p}\frac{A}{\ker\phi})^{****}$ such that $\xi(\ker\phi)=\{0\}.$ Now following the same course as in the previos lemma (step 2) we can find a bounded linear map $\theta:\frac{A}{\ker\phi}\rightarrow A^{****}$ such that $\tilde{\tilde{\phi}}\circ\theta(a+\ker\phi)=\phi(a)$ for each $a\in A.$ Choose $a_{0}$ in $A$ which $\phi(a_{0})=1.$ Set $m=\theta(a_{0}+\ker\phi)$. It is easy to see that $$am=\phi(a)m,\quad \tilde{\tilde{\phi}}(m)=1,\qquad (a\in A).$$ Applying Goldestine's theorem, we can find a bounded net $(m_{\alpha})$ in $A^{**}$ such that $am_{\alpha}-\phi(a)m_{\alpha}\xrightarrow{w^{*}} 0$ and $\tilde{\phi}(m_{\alpha})\rightarrow 1,$ for each $a\in A.$ On the other hand $(m_{\alpha})$ is a bounded net, therefore $(m_{\alpha})$ has a $w^{*}$-limit point, say $M$. It is easy to see that $aM=\phi(a)M$ and $\tilde{\phi}(M)=1$ for each $a\in A.$ Define $\eta:A\rightarrow (A\otimes_{p}A)^{**}$ by $\eta(a)=\phi(a)M\otimes M,$ for each $a\in A.$
It is easy to see that $\eta$ is a bounded linear map  such that $$\eta(ab)=a\cdot \eta(b)=\phi(b) \eta(a),\quad \tilde{\phi}\circ \pi^{**}_{A}\circ  \eta(a)=\phi(a)\quad (a,b\in A).$$
It follows that $A$ is left $\phi$-biflat.

Conversely, suppose that $A$ is left $\phi$-biflat. Since $A$ has a left approximate identity, we have $\overline{A\ker\phi}^{||\cdot||}=\ker\phi.$
So by previous lemma $A$ is left $\phi$-amenable. Applying \cite[Proposition 3.4]{kan} $A^{**}$ is left $\tilde{\phi}$-amenable. Thus there exists an element $m\in A^{****}$ such that $am=\phi(a)m$ and $\tilde{\tilde{\phi}}(m)=1$ for each $a\in A^{**}.$ Define $\gamma:A\rightarrow (A^{**}\otimes_{p}A^{**})^{**}$ by $\gamma(a)=\phi(a)m\otimes m,$ for each $a\in A.$
It is easy to see that $\gamma$ is a bounded linear map  such that $$\gamma(ab)=a\cdot \gamma(b)=\phi(b) \gamma(a),\quad \tilde{\phi}\circ \pi^{**}_{A^{**}}\circ  \gamma(a)=\phi(a)\quad (a,b\in A).$$
It follows that $A^{**}$ is left $\tilde{\phi}$-biflat. 
\end{proof}
\section{Applications to  Banach algebras related to a locally compact group}
 A linear subspace $S(G)$ of $L^{1}(G)$ is said to be
a Segal algebra on $G$ if it satisfies the following conditions
\begin{enumerate}
	\item [(i)] $S(G)$ is  dense    in $L^{1}(G)$,
	\item [(ii)]  $S(G)$ with a norm $||\cdot||_{S(G)}$ is
	a Banach space and $|| f||_{L^{1}(G)}\leq|| f||_{S(G)}$ for every
	$f\in S(G)$,
	\item [(iii)] for $f\in S(G)$ and $y\in G$, we have $L_{y}(f)\in S(G)$ the map $y\mapsto L_{y} (f)$ from $G$ into $S(G)$ is continuous, where
	$L_{y}(f)(x)=f(y^{-1}x)$,
	\item [(iv)] $||L_{y}(f)||_{S(G)}=||f||_{S(G)}$ for every $f\in
	S(G)$ and $y\in G$.
\end{enumerate}
For various examples of Segal algebras, we refer the reader  to \cite{rei}.

It is well-known that $S(G)$ always has a left approximate
identity.
For a Segal algebra $S(G)$ it  has been shown that
$$\Delta(S(G))=\{\phi_{|_{S(G)}}|\phi\in\Delta(L^{1}(G))\},$$ see
\cite[Lemma 2.2]{Ala}.
\begin{Theorem}
	Let $G$ be a locally compact group. Then the following statements are equivallent:
	\begin{enumerate}
		\item [(i)] $S(G)^{**}$ is  left $\tilde{\phi}$-biflat,
		\item [(ii)]  $S(G)$ is left $\phi$-biflat,
		\item [(iii)] $G$ is an amenable group. 
	\end{enumerate}
\end{Theorem}
\begin{proof}
$(i)\Rightarrow(ii)$
Let $S(G)^{**}$ be left $\tilde{\phi}$-biflat. Since $S(G)$  has a left approximate identity. Then by Theorem \ref{dual}, $S(G)$ is left $\phi$-biflat.

$(ii)\Rightarrow(iii)$
Suppose that $S(G)$ is left $\phi$-biflat. Since $S(G)$  has a left approximate identity, $\overline{S(G)\ker\phi}^{||\cdot||}=S(G).$ Applying Lemma \ref{main}, follows that $S(G)$ is left $\phi$-amenable. Now by \cite[Corollary 3.4]{Ala} $G$ is amenable.

$(iii)\Rightarrow(i)$
Let $G$ be amenable. So by \cite[Corollary 3.4]{Ala} $S(G)$ is left $\phi$-amenable. Thus $S(G)$ is left $\phi$-biflat. Using Theorem \ref{dual}, $S(G)^{**}$ is left $\tilde{\phi}$-biflat.
\end{proof}
Let $G$ be a locally compact group. Define $\mathcal{LA}(G)=L^{1}(G)\cap A(G)$, where $A(G)$ is the Fourier algebra over $G$. For $f\in \mathcal{LA}(G)$ put 
$$|||f|||=||f||_{L^{1}(G)}+||f||_{A(G)},$$ with this norm and the convoloution product $\mathcal{LA}(G)$ becomes a Banach algebra called Lebesgue-Fourier algebra. In fact $\mathcal{LA}(G)$ is a Segal algebra in $L^{1}(G)$, see \cite{gha lau}.
Following corollary is an easy consequence of previous theorem:
\begin{cor}
	Let $G$ be a locally compact group. Then the following statements are equivallent:
	\begin{enumerate}
		\item [(i)] $\mathcal{LA}(G)^{**}$ is  left $\tilde{\phi}$-biflat,
		\item [(ii)]  $\mathcal{LA}(G)$ is left $\phi$-biflat,
		\item [(iii)] $G$ is an amenable group. 
	\end{enumerate}
\end{cor}
Let $G$ be a locally compact group and let $\hat{G}$ be its dual
group, which consists of all non-zero continuous homomorphism
$\zeta:G\rightarrow \mathbb{T}$. It is well-known that
$\Delta(L^{1}(G))=\{\phi_{\zeta}:\zeta \in \hat{G}\}$, where
$\phi_{\zeta}(f)=\int_{G}\overline{\zeta(x)}f(x)dx$ and $dx$ is a
left Haar measure on $G$, for more details, see \cite[Theorem
23.7]{hew}.
\begin{cor}
Let $G$ be a locally compact group. Then the following statements are equivallent:
\begin{enumerate}
	\item [(i)] ${L^{1}(G)}^{**}$ is  left $\tilde{\phi}$-biflat,
	\item [(ii)]  $L^{1}(G)$ is left $\phi$-biflat,
	\item [(iii)] $G$ is an amenable group. 
\end{enumerate}
\end{cor}
\begin{proof}
	Clear.
\end{proof}
A discrete  semigroup $S$ is called  inverse semigroup, if for each $s\in S$
there exists an element $s^{*}\in S$ such that $ss^{*}s=s^{*}$ and
$s^{*}ss^{*}=s$.  There is a partial order on each
inverse semigroup $S$, that is,
$$s\leq t\Leftrightarrow s=ss^{*}t\quad (s,t \in S).$$
Let $(S,\leq)$ be an inverse semigroup. For each $s\in S$, set
$(x]=\{y\in S| \,y\leq x\}$. $S$ is called  uniformly locally finite if
$\sup\{|(x]|\,:x\in S\}<\infty$. Suppose that $S$ is an inverse
semigroup and $e\in E(S)$, where $E(S)$ is  the set of all idempotents of $S$. Then $G_{e}=\{s\in S|ss^{*}=s^{*}s=e\}$ is a
maximal subgroup of $S$ with respect to $e$.
An inverse semigroup $S$ is called
Clifford semigroup if for each $s\in S$ there exists $s^{*}\in S$
such that $ss^{*}=s^{*}s.$
See \cite{how} as a main reference of semigroup theory.
\begin{Proposition}
	Let $S=\cup_{e\in E(S)}G_{e}$ be a Clifford semigroup such that
	$E(S)$ is uniformly locally finite.
	Then the followings are equivalent:
	\begin{enumerate}
		\item [(i)] ${\ell^{1}(S)}^{**}$ is  left $\tilde{\phi}$-biflat, for each $\phi\in\Delta(\ell^{1}(S))$.
		\item [(ii)]  $\ell^{1}(S)$ is left $\phi$-biflat, for each $\phi\in\Delta(\ell^{1}(S))$.
		\item [(iii)] Each $G_{e}$ is an amenable group. 
		\item [(iv)] $\ell^{1}(S)$ is biflat.
	\end{enumerate}
\end{Proposition}
\begin{proof}
$(i)\Rightarrow(ii)$ 
	Suppose that $\ell^{1}(S)^{**}$ is left
	$\phi$-biflat for all $\phi\in\Delta(\ell^{1}(S))$. By \cite[Theorem
	2.16]{rams}, $\ell^{1}(S)\cong\ell^{1}-\oplus_{e\in
		E(S)}\ell^{1}(G_{e})$. Since each $\ell^{1}(G_{e})$ has an identity, $\ell^{1}(S)\cong\ell^{1}-\oplus_{e\in
		E(S)}\ell^{1}(G_{e})$ has  an approximate identity. Applying Theorem \ref{dual} gives that  $\ell^{1}(S)$ is left $\phi$-biflat.
	
	$(ii)\Rightarrow(iii)$ 
	Suppose that $\ell^{1}(S)$ is left $\phi$-biflat for each $\phi\in\Delta(\ell^{1}(S))$. Since $\ell^{1}(S)\cong\ell^{1}-\oplus_{e\in
		E(S)}\ell^{1}(G_{e})$ has  an approximate identity, Lemma \ref{main} implies that  $\ell^{1}(S)$ is left $\phi$-amenable for each $\phi\in\Delta(\ell^{1}(S))$. We know that each $\ell^{1}(G_{e})$ is a closed ideal in $\ell^{1}(S)$, so every non-zero multiplicative linear functional $\phi\in\Delta(\ell^{1}(G_{e}))$ can be extended to $\ell^{1}(S)$. Thus by \cite[Lemma 3.1]{kan} left $\phi$-amenability of  $\ell^{1}(S)$ implies that each $\ell^{1}(G_{e})$ is left $\phi$-amenable. Using \cite[Corollary 3.4]{Ala} each $G_{e}$ is amenable.
	
$(iv)\Rightarrow(i)$ 
It is clear by \cite[Theorem 3.7]{rams}.
\end{proof}
\begin{small}

\end{small}

\begin{thebibliography}{99}
\bibitem{Ala} M. Alaghmandan, R. Nasr Isfahani and M. Nemati, {\it Character amenability and contractibility of abstract Segal algebras}, Bull. Austral. Math. Soc, {\bf 82} (2010) 274-281.


\bibitem{rost ch} M. Essmaili,  M. Rostami,  M. Amini, {\it A characterization of biflatness of
	Segal algebras based on a character},  Glas. Mat. Ser. III 51(71) (2016), 45-58.

\bibitem{gha lau} F. Ghahramani and A. T. Lau, {\it Weak amenability of certain classes of Banach algebra without bounded approximate identity}, Math. Proc. Cambridge Philos. Soc {\bf 133}
(2002), 357-371.

\bibitem{gha} F. Ghahramani, R. J. Loy and G. A. Willis, {\it Amenability and weak
	amenability of second conjugate Banach algebras}, Proc. Amer. Math. Vol 124 (1996).

\bibitem{hew} E. Hewitt and K. A. Ross, Abstract harmonic analysis I,  Springer-Verlag, Berlin, (1963).


\bibitem{how} J. Howie, {\it Fundamental of semigroup theory}, London Math. Soc
Monographs, vol. {\bf 12}. Clarendon Press, Oxford (1995).

\bibitem{Hu} Z. Hu, M. S. Monfared and T. Traynor, {\it On character amenable Banach algebras}, Studia Math.
{\bf  193} (2009), 53–78.

\bibitem{jav} H. Javanshiri and M. Nemati, {\it Invariant $\phi$-means for abstract Segal algebras related to locally compact groups}, Bull. Belg. Math. Soc. Simon Stevin {\bf 25} (2018) 687-698.

\bibitem{kan} E. Kaniuth, A. T. Lau and J. Pym, {\it On $\phi$-amenability of Banach algebras}, Math. Proc. Camb. Philos. Soc. {\bf 44} (2008) 85-96.


\bibitem{rams} P. Ramsden, {\it Biflatness of semigroup algebras}, Semigroup Forum {\bf 79} (2009), 515-530.


\bibitem{rei} H. Reiter; {\it $L^{1}$-algebras and Segal Algebras}, Lecture Notes in Mathematics {\bf 231} (Springer, 1971).

\bibitem{run} V. Runde, {\it  Lectures on amenability}, Springer, New York, 2002.



\end{thebibliography}
\end{document}